\documentclass[12pt, reqno]{amsart}
\usepackage{amsmath, amsthm, amscd, amsfonts, amssymb, graphicx, color}
\usepackage[bookmarksnumbered, colorlinks, plainpages]{hyperref}
\input{mathrsfs.sty}

\hypersetup{colorlinks=true,linkcolor=red, anchorcolor=green,
citecolor=cyan, urlcolor=red, filecolor=magenta, pdftoolbar=true}

\newtheorem{theorem}{Theorem}[section]
\newtheorem{lemma}[theorem]{Lemma}

\newtheorem{corollary}[theorem]{Corollary}
\theoremstyle{definition}
\newtheorem{definition}[theorem]{Definition}

\theoremstyle{remark}

\numberwithin{equation}{section}

\begin{document}

\setcounter{page}{1}

\title[Characterization of numerical radius parallelism]
{Characterization of numerical radius parallelism in $C^*$-algebras}

\author[A. Zamani]{Ali Zamani}

\address{Department of Mathematics, Farhangian University, Iran}
\email{zamani.ali85@yahoo.com}


\begin{abstract}
Let $v(x)$ be the numerical radius of an element $x$ in a
$C^*$-algebra $\mathfrak{A}$. First, we prove several numerical
radius inequalities in $\mathfrak{A}$. Particularly,
we present a refinement of the triangle inequality for the numerical radius in $C^*$-algebras.
In addition, we show that if $x\in\mathfrak{A}$, then $v(x) = \frac{1}{2}\|x\|$
if and only if $\|x\| = \|\mbox{Re}(e^{i\theta}x)\| + \|\mbox{Im}(e^{i\theta}x)\|$
for all $\theta \in \mathbb{R}$.
Among other things, we introduce a new type of parallelism in
$C^*$-algebras based on numerical radius.
More precisely, we consider elements $x$
and $y$ of $\mathfrak{A}$ which satisfy
$v(x + \lambda x) = v(x) + v(y)$
for some complex unit $\lambda$.
We show that this relation can be characterized
in terms of pure states acting on $\mathfrak{A}$.
\end{abstract} \maketitle
\section{Introduction and preliminaries}
Let $\mathfrak{A}$ be a unital $C^*$-algebra with unit denoted by $e$.
We denote by $\mathcal{U}(\mathfrak{A})$ and $\mathcal{Z}(\mathfrak{A})$
the group of all unitary elements in $\mathfrak{A}$ and the centre of $\mathfrak{A}$, respectively.
For an element $x$ of $\mathfrak{A}$, we denote by
$\mbox{Re}(x) = \frac{1}{2}(x + x^*)$ and $\mbox{Im}(x) = \frac{1}{2i}(x - x^*)$
the real and the imaginary part of $x$.
Let $\mathfrak{A}'$ denote the dual space of $\mathfrak{A}$, and define the set of normalized states of
$\mathfrak{A}$ by
\begin{align*}
\mathcal{S}(\mathfrak{A}) = \{\varphi \in \mathfrak{A}':\, \varphi (e) = \|\varphi\| = 1\}.
\end{align*}
A linear functional $\varphi \in \mathfrak{A}'$ is said to be positive, and write $\varphi \geq 0$,
if $\varphi(x^*x) \geq 0$ for all $x\in \mathfrak{A}$ . Note that the set of
normalized states $\mathcal{S}(\mathfrak{A})$ is nothing but
\begin{align*}
\mathcal{S}(\mathfrak{A}) = \{\varphi \in \mathfrak{A}':\, \varphi \geq 0 \quad \mbox{and} \quad \varphi (e) = 1\}.
\end{align*}
Recall that a positive linear functional $\varphi$ on $\mathfrak{A}$ is said to be pure if for every positive
functional $\psi$ on $\mathfrak{A}$ satisfying $\psi(x^*x) \leq \varphi(x^*x)$
for all $x\in \mathfrak{A}$, there is a scalar $0 \leq \mu \leq 1$
such that $\psi = \mu \varphi$. The set of pure states on $\mathfrak{A}$
is denoted by $\mathcal{P}(\mathfrak{A})$.
The numerical range of an element $x\in \mathfrak{A}$ is
$V(x) = \{\varphi(x):\, \varphi \in \mathcal{S}(\mathfrak{A})\}.$
It is a nonempty compact and convex set of the complex plane $\mathbb{C}$,
and its maximum modulus is the numerical radius $v(x)$ of $x$; i.e.
$v(x) = \sup\{|z|: \, z\in V(x)\}$.
It is well-known that $v(\cdot)$ define a norm on $\mathfrak{A}$, which is equivalent
to the $C^*$-norm $\|\cdot\|$. In fact, the following inequalities are well-known:
\begin{align}\label{I.1.1}
\frac{1}{2}\|x\| \leq v(x)\leq \|x\| \qquad (x\in \mathfrak{A}).
\end{align}
It is a basic fact that the norm $v(\cdot)$ is self-adjoint (i.e., $v(x^*) = v(x)$ for
every $x\in \mathfrak{A}$) and also, if $x$ is normal, then $v(x) = \|x\|$.
Also, since $\mathcal{P}(\mathfrak{A})$ coincides with the set of
all extremal points of $\mathcal{S}(\mathfrak{A})$, thus
for every $x\in \mathfrak{A}$ we have
\begin{align*}
v(x) = \displaystyle{\sup_{\varphi \in \mathcal{S}(\mathfrak{A})}}|\varphi(x)|
= \displaystyle{\sup_{\varphi \in \mathcal{P}(\mathfrak{A})}}|\varphi (x)|.
\end{align*}

When $\mathfrak{A} = \mathbb{B}(\mathscr{H})$ is the $C^*$-algebra
of all bounded linear operators on a complex Hilbert space
$\big(\mathscr{H}, \langle \cdot, \cdot\rangle\big)$ and $T\in \mathbb{B}(\mathscr{H})$,
it well known that $V(T)$ is the closure of $W(T)$, the spatial numerical
range of $T$ defined by
$W(T) = \big\{\langle Tx, x\rangle:x\in \mathscr{H},\|x\| = 1\big\}.$
It is known as well that $W(T)$ is a nonempty bounded convex subset of $\mathbb{C}$
(not necessarily closed), and its supremum modulus, denoted by
$\omega(x) = \sup\{|z|: \, z\in W(T)\}$,
is called the spatial numerical radius of $T$ and coincides with $v(T)$.

For more material about the numerical radius and
other information on the basic theory of algebraic numerical range,
we refer the reader to \cite{B.D} and \cite{G.R}.
Some other related topics can be found in \cite{B.S, Dr, E.K, H.K.S, M.S, Sa, Sh, Z.2}.

Now, let $(\mathscr{X}, \|\cdot\|)$ be a normed space.
An element $x\in \mathscr{X}$ is said to be norm--parallel to another element
$y\in \mathscr{X}$ (see \cite{S, Z.M.1}), in short
$x\parallel y$, if
$\|x+\lambda y\|=\|x\|+\|y\|$ for some $\lambda\in\mathbb{T}$.
Here, as usual, $\mathbb{T}$ is the unit cycle of the complex plane $\mathbb{C}$.
In the context of continuous functions,
the well-known Daugavet equation $\|T + Id\| = \|T\| + 1$
is a particular case of parallelism. Here $Id$ denotes the identity function.
This property of a function, apart from being interesting in its own
right, arises naturally in problems dealing with best approximations in
function spaces; see \cite{We} and the references therein.
In the framework of inner product spaces, the norm--parallel relation
is exactly the usual vectorial parallel relation, that is,
$x\parallel y$ if and only if $x$ and $y$ are linearly dependent.
In the setting of normed linear spaces, two linearly
dependent vectors are norm--parallel, but the converse is false in general.

Some characterizations of the norm--parallelism for
operators on various Banach spaces and elements of an arbitrary Hilbert $C^*$-module
were given in \cite{B.C.M.W.Z, G, M.S.P, W, Z.1, Z.M.1, Z.M.2}.

Now, let us introduce a new type of parallelism in
$C^*$-algebras based on numerical radius.
\begin{definition}\label{D.1.2}
An element $x\in\mathfrak{A}$ is called the numerical radius parallel
to another element $y \in\mathfrak{A}$, denoted by $x\,{\parallel}_v \,y$, if
$v(x + \lambda x) = v(x) + v(y)$ for some $\lambda\in\mathbb{T}$.
\end{definition}
It is easy to see that the numerical radius parallelism is reflexive ($x\,{\parallel}_v \,x$),
symmetric ($x\,{\parallel}_v \,y$ if and only if
$y\,{\parallel}_v \,x$) and $\mathbb{R}$-homogenous
($x\,{\parallel}_v \,y \Rightarrow \alpha x\,{\parallel}_v \,\beta y$
for all $\alpha, \beta \in \mathbb{R}$)).
Notice that two linearly dependent elements are numerical
radius parallel. The converse is however not true, in general.

The organization of this paper will be as follows.
Inspired by the numerical radius inequalities of bounded linear operators
in \cite{A.K.1}, \cite{A.K.2}, \cite{K.1}, \cite{K.2}, \cite{K.3}, \cite{K.M.Y}, \cite{Y}
and by using some ideas of them,
we firstly state a useful characterization
of the numerical radius for elements of a $C^*$-algebra, as follows:
\begin{align*}
v(x) = \displaystyle{\sup_{\theta \in \mathbb{R}}}\|\mbox{Re}(e^{i\theta}x)\|.
\end{align*}
We then apply it to prove that several numerical
radius inequalities in $C^*$-algebras.
Moreover, we give new improvements of the inequalities (\ref{I.1.1}).

We also give an expression of $v(x)$ in terms of the real and imaginary
parts of $x\in \mathfrak{A}$, as follows:
\begin{align*}
v(x) = \displaystyle{\sup_{\alpha^2 + \beta^2 = 1}}\big\|\alpha \mbox{Re}(x) + \beta \mbox{Im}(x)\big\|.
\end{align*}
Particularly, then we show that if $x\in\mathfrak{A}$, then $v(x) = \frac{1}{2}\|x\|$
if and only if $\|x\| = \|\mbox{Re}(e^{i\theta}x)\| + \|\mbox{Im}(e^{i\theta}x)\|$
for all $\theta \in \mathbb{R}$.
Our results generalize recent numerical
radius inequalities of bounded linear operators due to Kittaneh et al.
\cite{K.2, K.M.Y, Y}.

In addition, we present a refinement of the triangle inequality
for the numerical radius in $C^*$-algebras.
We then apply it to give a necessary condition for the numerical radius parallelism.
Furthermore, for two elements $x$ and $y$ in $\mathfrak{A}$
we show that $x\,{\parallel}_v \,y$ if and only if there exists a pure state
$\varphi$ on $\mathfrak{A}$ such that $|\varphi(x)\varphi(y)| = v(x)v(y)$.
Finally, we prove that if $c\in \mathcal{Z}(\mathfrak{A})\cap \mathcal{U}(\mathfrak{A})$,
then $x\,{\parallel}_v \,y$ holds exactly when $cx\,{\parallel}_v \,cy$.
\section{Main results}
We start our work with the following lemma.
\begin{lemma}\label{L.2.1}
Let $\mathfrak{A}$ be a $C^*$-algebra and let $\varphi$ be a state over $\mathfrak{A}$.
For $x\in \mathfrak{A}$ the following statements hold.
\begin{itemize}
\item[(i)] $\displaystyle{\sup_{\theta \in \mathbb{R}}}\big|\mbox{Re}\big(e^{i\theta}\varphi(x)\big)\big| = |\varphi(x)|$.
\item[(ii)] $\displaystyle{\sup_{\theta \in \mathbb{R}}}\big|\mbox{Im}\big(e^{i\theta}\varphi(x)\big)\big| = |\varphi(x)|$.
\end{itemize}
\end{lemma}
\begin{proof}
We may assume that $\varphi(x) \neq 0$ otherwise (i) and (ii) trivially hold.

(i) Put $e^{i{\theta}_0} = \frac{\overline{\varphi(x)}}{|\varphi(x)|}$.
Then we have
\begin{align*}
|\varphi(x)| = \big|\mbox{Re}\big(e^{i{\theta}_0}\varphi(x)\big)\big|
\leq \sup_{\theta \in \mathbb{R}}\big|\mbox{Re}\big(e^{i\theta}\varphi(x)\big)\big|
\leq \sup_{\theta \in \mathbb{R}}\big|e^{i\theta}\varphi(x)\big| = |\varphi(x)|,
\end{align*}
and hence $|\varphi(x)| = \sup_{\theta \in \mathbb{R}}\big|\mbox{Re}\big(e^{i\theta}\varphi(x)\big)\big|$.

(ii) By replacing $x$ in (i) by $ix$, we obtain
\begin{align*}
\sup_{\theta \in \mathbb{R}}\big|\mbox{Im}\big(e^{i\theta}\varphi(x)\big)\big|
= \sup_{\theta \in \mathbb{R}}\big|\mbox{Re}\big(e^{i\theta}\varphi(ix)\big)\big|
= |\varphi(ix)| = |\varphi(x)|.
\end{align*}
\end{proof}
Now, we are in a position to state two useful characterizations
of the numerical radius for elements of a $C^*$-algebra.
\begin{theorem}\label{T.2.2}
Let $\mathfrak{A}$ be a $C^*$-algebra.
For $x\in \mathfrak{A}$ the following statements hold.
\begin{itemize}
\item[(i)] $\displaystyle{\sup_{\theta \in \mathbb{R}}}\|\mbox{Re}(e^{i\theta}x)\| = v(x)$.
\item[(ii)] $\displaystyle{\sup_{\theta \in \mathbb{R}}}\|\mbox{Im}(e^{i\theta}x)\| = v(x)$.
\end{itemize}
\end{theorem}
\begin{proof}
(i) Since $\mbox{Re}(e^{i\theta}x)$ is self adjoint for any $\theta \in \mathbb{R}$,
we have
\begin{align*}
\|\mbox{Re}(e^{i\theta}x)\| = v(\mbox{Re}(e^{i\theta}x)).
\end{align*}
Therefore, we get
\begin{align*}
\sup_{\theta \in \mathbb{R}}\|\mbox{Re}(e^{i\theta}x)\| &= \sup_{\theta \in \mathbb{R}}v(\mbox{Re}(e^{i\theta}x))
\\& = \sup_{\theta \in \mathbb{R}}\sup_{\varphi\in \mathcal{S}(\mathfrak{A})}|\varphi\big(\mbox{Re}(e^{i\theta}x)\big)|
\\& = \sup_{\theta \in \mathbb{R}}\sup_{\varphi\in \mathcal{S}(\mathfrak{A})}|\mbox{Re}(e^{i\theta}\varphi(x))|
\\& = \sup_{\varphi\in \mathcal{S}(\mathfrak{A})}\sup_{\theta \in \mathbb{R}}|\mbox{Re}(e^{i\theta}\varphi(x))|
\\& = \sup_{\varphi\in \mathcal{S}(\mathfrak{A})}|\varphi(x)|
\qquad\big(\mbox{by Lemma \ref{L.2.1} (i)}\big)
\\& = v(x).
\end{align*}
Thus $\displaystyle{\sup_{\theta \in \mathbb{R}}}\|\mbox{Re}(e^{i\theta}x)\| = v(x)$.

(ii) By replacing $x$ in (i) by $ix$, we reach that
\begin{align*}
\sup_{\theta \in \mathbb{R}}\|\mbox{Im}(e^{i\theta}x)\|
= \sup_{\theta \in \mathbb{R}}\|\mbox{Re}(e^{i\theta}(ix))\|
= v(ix) = v(x).
\end{align*}
\end{proof}
Recall that the Crawford number of $T\in\mathbb{B}(\mathscr{H})$ is defined by
\begin{align*}
c(T) = \inf \{|\langle Tx, x\rangle|:x\in \mathscr{H},\|x\| =1\}.
\end{align*}
This concept is useful in studying linear operators (e.g., see \cite{A.K.1, Z.2}, and their references).
The Crawford number of $z\in \mathfrak{A}$ can be defined by
\begin{align*}
c(z) = \inf\{|\varphi(z)|: \, \varphi \in \mathcal{S}(\mathfrak{A})\}.
\end{align*}
In the following theorem, we give new improvement of the inequalities (\ref{I.1.1}).
\begin{theorem}\label{T.2.3}
Let $\mathfrak{A}$ be a $C^*$-algebra.
For $x\in \mathfrak{A}$ the following statements hold.
\begin{itemize}
\item[(i)] $\frac{1}{2}\|x\| \leq \frac{1}{2}\sqrt{\big\|\,|x|^2 + |x^*|^2\big\| + 2c(x^2)} \leq v(x)$.
\item[(ii)] $v(x) \leq \frac{1}{2}\sqrt{\big\|\,|x|^2 + |x^*|^2\big\| + 2v(x^2)}
\leq \frac{1}{2}\big(\|x\| + {\|x^2\|}^{\frac{1}{2}}\big) \leq \|x\|$.
\end{itemize}
\end{theorem}
\begin{proof}
(i) Let $x\in \mathfrak{A}$. By \cite[Theorem 3.3.6]{Mur} there is a state $\varphi$ over $\mathfrak{A}$ such that
\begin{align}\label{I.2.3.1}
\varphi\big(|x|^2 + |x^*|^2\big) = \big\|\,|x|^2 + |x^*|^2\big\|.
\end{align}
Let ${\theta}_0$ be a real number such that $|\varphi(x^2)| = e^{2i{\theta}_0}\varphi(x^2)$. Then, by Theorem \ref{T.2.2} (i), we have
\begin{align*}
v(x) \geq \|\mbox{Re}(e^{i{\theta}_0}x)\| &= \frac{1}{2}\|e^{i{\theta}_0}x + e^{-i{\theta}_0}x^*\|
\\& = \frac{1}{2}\sqrt{\big\|\big(e^{i{\theta}_0}x + e^{-i{\theta}_0}x^*\big)\big(e^{i{\theta}_0}x + e^{-i{\theta}_0}x^*\big)^*\big\|}
\\& = \frac{1}{2}\sqrt{\big\|\,|x|^2 + |x^*|^2 + 2\mbox{Re}(e^{2i{\theta}_0}x^2)\big\|}
\\& \geq \frac{1}{2}\sqrt{\big|\varphi\big(|x|^2 + |x^*|^2 + 2\mbox{Re}(e^{2i{\theta}_0}x^2)\big)\big|}
\\& = \frac{1}{2}\sqrt{\big|\varphi\big(|x|^2 + |x^*|^2\big) + 2\mbox{Re}\big(e^{2i{\theta}_0}\varphi(x^2)\big)\big|}
\\& = \frac{1}{2}\sqrt{\big\|\,|x|^2 + |x^*|^2\big\| + 2|\varphi(x^2)|} \qquad\big(\mbox{by (\ref{I.2.3.1})}\big)
\\& \geq \frac{1}{2}\sqrt{\big\|\,|x|^2 + |x^*|^2\big\| + 2c(x^2)} \geq \frac{1}{2}\|x\|,
\end{align*}
which proves the inequalities in (i).

(ii) By Theorem \ref{T.2.2} (i), as in the proof of (i) we get
\begin{align*}
v(x) &= \sup_{\theta \in \mathbb{R}}\|\mbox{Re}(e^{i\theta}x)\|
\\& = \frac{1}{2}\sup_{\theta \in \mathbb{R}}\sqrt{\big\|\,|x|^2 + |x^*|^2 + 2\mbox{Re}(e^{2i\theta}x^2)\big\|}
\\& \leq \frac{1}{2}\sup_{\theta \in \mathbb{R}}\sqrt{\big\|\,|x|^2 + |x^*|^2\| + 2\|\mbox{Re}(e^{2i\theta}x^2)\big\|}
\\& \leq \frac{1}{2}\sqrt{\big\|\,|x|^2 + |x^*|^2\| + 2\sup_{\theta \in \mathbb{R}}\|\mbox{Re}(e^{2i\theta}x^2)\big\|}
\\& = \frac{1}{2}\sqrt{\big\|\,|x|^2 + |x^*|^2\big\| + 2v(x^2)}
\\& \leq \frac{1}{2}\sqrt{\|x\|^2 + \|x^2\| + 2v(x^2)}
\\& \leq \frac{1}{2}\sqrt{\|x\|^2 + 3\|x^2\|} \qquad\big(\mbox{by (\ref{I.1.1})}\big)
\\& \leq \frac{1}{2}\sqrt{\|x\|^2 + 2\|x\|{\|x^2\|}^{\frac{1}{2}} + \|x^2\|}
\qquad\big(\mbox{since $\|x^2\| = {\|x^2\|}^{\frac{1}{2}}{\|x^2\|}^{\frac{1}{2}} \leq \|x\|{\|x^2\|}^{\frac{1}{2}}$}\big)
\\& = \frac{1}{2}\big(\|x\| + {\|x^2\|}^{\frac{1}{2}}\big) \leq \|x\|,
\end{align*}
which proves the inequalities in (ii).
\end{proof}
As a consequence of Theorem \ref{T.2.3}, we have the following result.
\begin{corollary}\label{C.2.4}
Let $\mathfrak{A}$ be a $C^*$-algebra. If $x\in \mathfrak{A}$ is such that $x^2 = 0$,
then $v(x) = \frac{1}{2}\|x\|$.
\end{corollary}
\begin{proof}
Since $x^2 = 0$, by Theorem \ref{T.2.3} (ii), we obtain
$v(x) \leq \frac{1}{2}\big(\|x\| + {\|x^2\|}^{\frac{1}{2}}\big) = \frac{1}{2}\|x\|$.
We also have that $\frac{1}{2}\|x\| \leq v(x)$ for every $x\in \mathfrak{A}$.
Thus $v(x) = \frac{1}{2}\|x\|$.
\end{proof}
The following result is another consequence of Theorem \ref{T.2.3}.
\begin{corollary}\label{C.2.5}
Let $\mathfrak{A}$ be a $C^*$-algebra. If $x\in \mathfrak{A}$ is such that $v(x) = \|x\|$,
then $\|x^2\| = {\|x\|}^2$.
\end{corollary}
\begin{proof}
It follows from Theorem \ref{T.2.2} (ii) that $v(x) = \|x\|$ implies
$\|x\| \leq \frac{1}{2}\big(\|x\| + {\|x^2\|}^{\frac{1}{2}}\big) \leq \|x\|$.
Thus $\|x\| = {\|x^2\|}^{\frac{1}{2}}$, or equivalently $\|x^2\| = {\|x\|}^2$.
\end{proof}
In the following theorem we give an expression of $v(x)$
in terms of the real and imaginary parts of $x\in \mathfrak{A}$.
\begin{theorem}\label{T.2.6}
Let $\mathfrak{A}$ be a $C^*$-algebra and let $x\in \mathfrak{A}$.
Then for $\alpha , \beta \in \mathbb{R}$, the following statements hold.
\begin{itemize}
\item[(i)] $\displaystyle{\sup_{\alpha^2 + \beta^2 = 1}}\big\|\alpha \mbox{Re}(x) + \beta \mbox{Im}(x)\big\| = v(x)$.
\item[(ii)] $\max\big\{\|\mbox{Re}(x)\|, \|\mbox{Im}(x)\|\big\}\leq v(x)$.
\end{itemize}
\end{theorem}
\begin{proof}
(i) Let $\theta \in \mathbb{R}$. Put $\alpha = \cos \theta$ and $\beta = -\sin \theta$.
We have
\begin{align*}
\mbox{Re}(e^{i\theta}x) &= \frac{e^{i\theta}x + e^{-i\theta}x^*}{2}
\\& = \frac{(\cos \theta + i\sin \theta)x + (\cos \theta -i \sin \theta)x^*}{2}
\\& = (\cos \theta)\frac{x + x^*}{2} - (\sin \theta)\frac{x - x^*}{2i}
\\& = \alpha \mbox{Re}(x) + \beta \mbox{Im}(x).
\end{align*}
Therefore
\begin{align*}
\sup_{\theta \in \mathbb{R}}\|\mbox{Re}(e^{i\theta}x)\|
= \sup_{\alpha^2 + \beta^2 = 1}\big\|\alpha \mbox{Re}(x) + \beta \mbox{Im}(x)\big\|,
\end{align*}
and hence by Theorem \ref{T.2.2} (i) we obtain
$v(x) = \displaystyle{\sup_{\alpha^2 + \beta^2 = 1}}\big\|\alpha \mbox{Re}(x) + \beta \mbox{Im}(x)\big\|$.

(ii) By setting $(\alpha, \beta) = (1, 0)$ and $(\alpha, \beta) =(0, 1)$ in (i), we get
$\|\mbox{Re}(x)\| \leq v(x)$ and $\|\mbox{Im}(x)\| \leq v(x)$. Thus $\max\big\{\|\mbox{Re}(x)\|, \|\mbox{Im}(x)\|\big\}\leq v(x)$.
\end{proof}
In the next result, we obtain a necessary
and sufficient condition for $v(x) = \frac{1}{2}\|x\|$ to hold.
We will need the following lemma.
\begin{lemma}\cite[Corollary 4.4]{Z.M.1}\label{L.2.7}
Let $\mathfrak{A}$ be a $C^*$-algebra and let $x, y\in \mathfrak{A}$.
Then the following statements are equivalent:
\begin{itemize}
\item[(i)] $x\parallel y$.
\item[(ii)] There exists a state $\varphi$ over $\mathfrak{A}$ such that $|\varphi(x^*y)| = \|x\|\|y\|$.
\end{itemize}
\end{lemma}
\begin{theorem}\label{T.2.8}
Let $\mathfrak{A}$ be a $C^*$-algebra and let $x\in \mathfrak{A}$.
Then the following statements are equivalent:
\begin{itemize}
\item[(i)] $v(x) = \frac{1}{2}\|x\|$.
\item[(ii)] $\|x\| = \|\mbox{Re}(e^{i\theta}x)\| + \|\mbox{Im}(e^{i\theta}x)\|$
for all $\theta \in \mathbb{R}$.
\end{itemize}
\end{theorem}
\begin{proof}
(i)$\Rightarrow$(ii) Suppose that $v(x) = \frac{1}{2}\|x\|$.
Then for any $\theta \in \mathbb{R}$, we have
\begin{align*}
\|x\| = \|e^{i\theta}x\| &= \|\mbox{Re}(e^{i\theta}x) + i\mbox{Im}(e^{i\theta}x)\|
\\&\leq \|\mbox{Re}(e^{i\theta}x)\| + \|\mbox{Im}(e^{i\theta}x)\|
\\&\leq 2\max\big\{\|\mbox{Re}(e^{i\theta}x)\|, \|\mbox{Im}(e^{i\theta}x)\|\big\}
\\& \leq 2v(e^{i\theta}x)
\qquad\big(\mbox{by Theorem \ref{T.2.6} (ii)}\big)
\\& = 2v(x) = \|x\|,
\end{align*}
and hence $\|x\| = \|\mbox{Re}(e^{i\theta}x)\| + \|\mbox{Im}(e^{i\theta}x)\|$.

(ii)$\Rightarrow$(i) Suppose (ii) holds.
Thus for all $\theta \in \mathbb{R}$,
\begin{align*}
\|\mbox{Re}(e^{i\theta}x) + i\mbox{Im}(e^{i\theta}x)\| = \|\mbox{Re}(e^{i\theta}x)\| + \|\mbox{Im}(e^{i\theta}x)\|,
\end{align*}
so $\mbox{Re}(e^{i\theta}x) \parallel \mbox{Im}(e^{i\theta}x)$.
By Lemma \ref{L.2.7}, there exists a state $\varphi$ over $\mathfrak{A}$ such that
\begin{align*}
\Big|\varphi\Big(\big(\mbox{Re}(e^{i\theta}x)\big)^*\mbox{Im}(e^{i\theta}x)\Big)\Big|
= \|\mbox{Re}(e^{i\theta}x)\|\,\|\mbox{Im}(e^{i\theta}x)\|,
\end{align*}
and hence
\begin{align*}
\big|\varphi\big(\mbox{Re}(e^{i\theta}x)\mbox{Im}(e^{i\theta}x)\big)\big|
= \|\mbox{Re}(e^{i\theta}x)\|\,\|\mbox{Im}(e^{i\theta}x)\|.
\end{align*}
From this it follows that $v\big(\mbox{Re}(e^{i\theta}x)\mbox{Im}(e^{i\theta}x)\big)
= \|\mbox{Re}(e^{i\theta}x)\|\,\|\mbox{Im}(e^{i\theta}x)\|$,
so by Theorem \ref{T.2.2} (ii) we reach that
\begin{align}\label{T.2.8.1}
\|\mbox{Re}(e^{i\theta}x)\|\,\|\mbox{Im}(e^{i\theta}x)\|
= \big\|\mbox{Im}\big((\mbox{Re}(e^{i\theta}x)\mbox{Im}(e^{i\theta}x)\big)\big\|.
\end{align}
On the other hand,
\begin{align*}
\mbox{Im}\big(\mbox{Re}(e^{i\theta}x)\mbox{Im}(e^{i\theta}x)\big)
&= \mbox{Im}\Big((\frac{e^{i\theta}x + e^{-i\theta}x^*}{2})(\frac{e^{i\theta}x - e^{-i\theta}x^*}{2i})\Big)
\\& = \mbox{Im}\Big(\frac{e^{2i\theta}x^2 - e^{-2i\theta}{x^*}^2 -xx^* + x^*x}{4i}\Big)
\\& = \frac{1}{2i}\Big\{\frac{e^{2i\theta}x^2 - e^{-2i\theta}{x^*}^2 -xx^* + x^*x}{4i}
\\& \qquad \qquad \qquad - \frac{e^{-2i\theta}{x^*}^2 - e^{2i\theta}x^2 -xx^* + x^*x}{-4i}\Big\}
\\& = \frac{xx^* - x^*x}{4}
= \mbox{Im}\big(\mbox{Re}(x)\mbox{Im}(x)\big),
\end{align*}
and by (\ref{T.2.8.1}) we get
\begin{align}\label{T.2.8.2}
\|\mbox{Re}(e^{i\theta}x)\|\,\|\mbox{Im}(e^{i\theta}x)\|
= \big\|\mbox{Im}\big(\mbox{Re}(x)\mbox{Im}(x)\big)\big\|.
\end{align}

Thus for all $\theta \in \mathbb{R}$, by (ii) and (\ref{T.2.8.2}) we obtain
\begin{align}\label{T.2.8.3}
\|\mbox{Re}(e^{i\theta}x)\| = \frac{\|x\| + \sqrt{\|x\|^2
- 4\big\|\mbox{Im}\big(\mbox{Re}(x)\mbox{Im}(x)\big)\big\|}}{2}
\end{align}
and
\begin{align}\label{T.2.8.4}
\|\mbox{Im}(e^{i\theta}x)\| = \frac{\|x\| - \sqrt{\|x\|^2
- 4\big\|\mbox{Im}\big(\mbox{Re}(x)\mbox{Im}(x)\big)\big\|}}{2}.
\end{align}
Since
\begin{align*}
\mbox{Re}(e^{i\theta}x)
= \frac{(\cos \theta + i\sin \theta)x + (\cos \theta - i\sin \theta)x^*}{2}
= \cos \theta \mbox{Re}(x) - \sin \theta\mbox{Im}(x)
\end{align*}
and
\begin{align*}
\mbox{Im}(e^{i\theta}x)
= \frac{(\cos \theta + i\sin \theta)x - (\cos \theta - i\sin \theta)x^*}{2i}
= \sin \theta\mbox{Re}(x) + \cos \theta\mbox{Im}(x)
\end{align*}
So, from relations (\ref{T.2.8.3}) and (\ref{T.2.8.4}), we conclude that
the functions $\|\mbox{Re}(e^{i\theta}x)\|, \|\mbox{Im}(e^{i\theta}x)\|$
are continuous on $\theta \in \mathbb{R}$ and therefore they must be constant, i.e.,
\begin{align*}
\|\mbox{Re}(e^{i\theta}x)\| = \|\mbox{Im}(e^{i\theta}x)\| = \frac{1}{2}\|x\| \qquad (\theta \in \mathbb{R}).
\end{align*}
Thus $\displaystyle{\sup_{\theta \in \mathbb{R}}}\|\mbox{Re}(e^{i\theta}x)\| = \frac{1}{2}\|x\|$.
Now, by Theorem \ref{T.2.2} (i) we conclude that $v(x) = \frac{1}{2}\|x\|$.
\end{proof}
Another, we present new improvement of the inequalities (\ref{I.1.1}).
\begin{theorem}\label{T.2.9}
Let $\mathfrak{A}$ be a $C^*$-algebra. For $x\in \mathfrak{A}$ the following statements hold.
\begin{itemize}
\item[(i)] $\frac{1}{2}\|x\| \leq \frac{1}{2}\sqrt{\|x^*x + xx^*\|}\leq v(x)$.
\item[(ii)] $v(x) \leq \frac{1}{\sqrt{2}}\sqrt{\|x^*x + xx^*\|} \leq \|x\|$.
\end{itemize}
\end{theorem}
\begin{proof}
(i) Let $x\in \mathfrak{A}$. Clearly, $\frac{1}{2}\|x\| \leq \frac{1}{2}\sqrt{\|x^*x + xx^*\|}$.
But, by simple computations,
\begin{align*}
x^*x + xx^* = 2\mbox{Re}^2(x) + 2\mbox{Im}^2(x).
\end{align*}
Consequently, by Theorem \ref{T.2.6} (ii) we get
\begin{align*}
\frac{1}{2}\sqrt{\|x^*x + xx^*\|} &= \frac{1}{2}\sqrt{\big\|2\mbox{Re}^2(x) + 2\mbox{Im}^2(x)\big\|}
\\& \leq \frac{1}{2}\sqrt{2\|\mbox{Re}(x)\|^2 + 2\|\mbox{Im}(x)\|^2}
\\& \leq \frac{1}{2}\sqrt{2v^2(x) + 2v^2(x)} = v(x)
\end{align*}
Therefore $\frac{1}{2}\sqrt{\|x^*x + xx^*\|}\leq v(x)$.

(ii) Obviously, $\frac{1}{\sqrt{2}}\sqrt{\|x^*x + xx^*\|} \leq \|x\|$.
Now, let $\pi: \mathfrak{A} \rightarrow \mathbb{B}(\mathscr{H})$ be a non-degenerate faithful representation
of $\mathfrak{A}$ on some Hilbert space $\mathscr{H}$ (see \cite[Theorem 2.6.1]{Dix}).
Let $\alpha, \beta \in \mathbb{R}$ satisfy $\alpha^2 + \beta^2 = 1$.
Then for any unit vector $\xi \in \mathscr{H}$, we have
\begin{align*}
\big\|\pi\big(\alpha \mbox{Re}(x) + \beta \mbox{Im}(x)\big)\xi\big\| &
=\left\|\begin{bmatrix}
\pi(\mbox{Re}(x)) & \pi(\mbox{Im}(x))\\
0 & 0
\end{bmatrix}
\begin{bmatrix}
\alpha \xi\\
\beta \xi
\end{bmatrix}
\right\|
\\&\leq \left\|\begin{bmatrix}
\pi(\mbox{Re}(x)) & \pi(\mbox{Im}(x))\\
0 & 0
\end{bmatrix}
\right\|
\\& = \left\|\begin{bmatrix}
\mbox{Re}(\pi(x)) & \mbox{Im}(\pi(x))\\
0 & 0
\end{bmatrix}
\right\|
\,\,\quad(\mbox{since $\pi$ is representation})
\\& = {\left\|\begin{bmatrix}
\mbox{Re}(\pi(x)) & \mbox{Im}(\pi(x))\\
0 & 0
\end{bmatrix}
\begin{bmatrix}
\mbox{Re}(\pi(x)) & 0\\
\mbox{Im}(\pi(x)) & 0
\end{bmatrix}
\right\|}^{\frac{1}{2}}
\\&= {\big\|\mbox{Re}^2(\pi(x)) + \mbox{Im}^2(\pi(x))\big\|}^{\frac{1}{2}}
\\&= \frac{1}{\sqrt{2}}{\big\|\pi(x)^*\pi(x) + \pi(x)\pi(x)^*\big\|}^{\frac{1}{2}}
\\&= \frac{1}{\sqrt{2}}{\big\|\pi(x^*x + xx^*)\big\|}^{\frac{1}{2}}
\\&= \frac{1}{\sqrt{2}}{\|x^*x + xx^*\|}^{\frac{1}{2}}.
\qquad(\mbox{since $\pi$ is isometric})
\end{align*}
Hence we have $\big\|\pi(\alpha \mbox{Re}(x) + \beta \mbox{Im}(x))\xi\big\| \leq \frac{1}{\sqrt{2}}\sqrt{\|x^*x + xx^*\|}$
and so by taking the supremum over all $\xi \in \mathscr{H}$ we obtain
$\|\pi(\alpha \mbox{Re}(x) + \beta \mbox{Im}(x))\big\| \leq \frac{1}{\sqrt{2}}\sqrt{\|x^*x + xx^*\|}$.
From this it follows that $\big\|\alpha \mbox{Re}(x) + \beta \mbox{Im}(x)\big\| \leq \frac{1}{\sqrt{2}}\sqrt{\|x^*x + xx^*\|}$
and hence
\begin{align*}
\displaystyle{\sup_{\alpha^2 + \beta^2 = 1}}\big\|\alpha \mbox{Re}(x) + \beta \mbox{Im}(x)\big\|
\leq \frac{1}{\sqrt{2}}\sqrt{\|x^*x + xx^*\|}.
\end{align*}
Now, by Theorem \ref{T.2.6} (i) we conclude that $v(x) \leq \frac{1}{\sqrt{2}}\sqrt{\|x^*x + xx^*\|}$.
\end{proof}
In what follows, $r(x)$ stands for the spectral radius of
an arbitrary element $x$ in a $C^*$-algebra $\mathfrak{A}$.
It is well known that for every $x\in \mathfrak{A}$, we
have $r(x) \leq \|x\|$ and that equality holds in this inequality if $x$ is normal.
In the following lemma we obtain a spectral radius inequality for sums of elements
in $C^*$-algebras.
\begin{lemma}\label{L.2.10}
Let $\mathfrak{A}$ be a $C^*$-algebra and let $z, w\in \mathfrak{A}$. Then
\begin{align*}
r(z + w) \leq \frac{1}{2} \Big(\|z\| + \|w\|
+ \sqrt{(\|z\| - \|w\|)^2 + 4 \min\{\|zw\|, \|wz\|\}}\Big).
\end{align*}
\end{lemma}
\begin{proof}
We first recall that \cite[Corollary 1]{K.1} tells us that
\begin{align}\label{L.2.10.1}
r(T + S) \leq \frac{1}{2} \Big(\|T\| + \|S\|
+ \sqrt{(\|T\| - \|S\|)^2 + 4 \min\{\|TS\|, \|ST\|\}}\Big),
\end{align}
for all bounded linear operators $T, S$ that acting on a Hilbert space.
Now, let $\pi: \mathfrak{A} \rightarrow \mathbb{B}(\mathscr{H})$ be a non-degenerate faithful representation
of $\mathfrak{A}$ on some Hilbert space $\mathscr{H}$ (see \cite[Theorem 2.6.1]{Dix}).
Since $\pi$ is isometric, by letting $T = \pi(z)$ and $S = \pi(w)$ in (\ref{L.2.10.1}), we obtain
\begin{align*}
r(z + w) &= r\big(\pi(z) + \pi(w\big))
\\ &\leq \frac{1}{2} \Big(\|\pi(z)\| + \|\pi(w)\|
\\& \quad \qquad + \sqrt{(\|\pi(z)\| - \|\pi(w)\|)^2 + 4 \min\{\|\pi(z)\pi(w)\|, \|\pi(w)\pi(z)\|\}}\Big)
\\& = \frac{1}{2} \Big(\|z\| + \|w\|
+ \sqrt{(\|z\| - \|w\|)^2 + 4 \min\{\|zw\|, \|wz\|\}}\Big),
\end{align*}
and the statement is proved.
\end{proof}
Now, we present a refinement of the triangle inequality
for the numerical radius in $C^*$-algebras.
\begin{theorem}\label{T.2.11}
Let $\mathfrak{A}$ be a $C^*$-algebra.
For $x, y\in \mathfrak{A}$ the following statements hold.
\begin{itemize}
\item[(i)] \begin{align*}
v(x + y) &\leq \frac{1}{2} \big(v(x) + v(y)\big)
\\& \qquad + \frac{1}{2}\sqrt{\big(v(x) - v(y)\big)^2 +
4 \sup_{\theta \in \mathbb{R}}\big\|\mbox{Re}(e^{i\theta}x)\mbox{Re}(e^{i\theta} y)\big\|}
\\& \leq v(x) + v(y). \end{align*}
\item[(ii)] \begin{align*}
v(x + y) &\leq \frac{1}{2} \big(v(x) + v(y)\big)
\\& \qquad + \frac{1}{2}\sqrt{\big(v(x) - v(y)\big)^2 +
4 \sup_{\theta \in \mathbb{R}}\big\|\mbox{Im}(e^{i\theta}x)\mbox{Im}(e^{i\theta} y)\big\|}
\\& \leq v(x) + v(y). \end{align*}
\end{itemize}
\end{theorem}
\begin{proof}
(i) Since $\mbox{Re}(e^{i\theta}(x + y))$ is self adjoint for any $\theta \in \mathbb{R}$,
we have
\begin{align*}
\|\mbox{Re}(e^{i\theta}(x + y))\| = r\big(\mbox{Re}(e^{i\theta}(x + y))\big).
\end{align*}
So, by letting $z = \mbox{Re}(e^{i\theta}x)$
and $w = \mbox{Re}(e^{i\theta}y)$ in Lemma \ref{L.2.10}, we obtain
\begin{align*}
\|\mbox{Re}(e^{i\theta}(x + y))\|& = r\big(\mbox{Re}(e^{i\theta}(x + y))\big)
\\&= r\big(\mbox{Re}(e^{i\theta}x) + \mbox{Re}(e^{i\theta}y)\big)
\\& \leq \frac{1}{2} \Big(\|\mbox{Re}(e^{i\theta}x)\| + \|\mbox{Re}(e^{i\theta}y)\|
\\& \qquad + \sqrt{(\|\mbox{Re}(e^{i\theta}x)\| - \|\mbox{Re}(e^{i\theta}y)\|)^2
+ 4 \|\mbox{Re}(e^{i\theta}x)\mbox{Re}(e^{i\theta}y)\|}\Big)
\\& = \left\|\begin{bmatrix}
\|\mbox{Re}(e^{i\theta}x)\| & \sqrt{\|\mbox{Re}(e^{i\theta}x)\mbox{Re}(e^{i\theta}y)\|}
\\ \sqrt{\|\mbox{Re}(e^{i\theta}x)\mbox{Re}(e^{i\theta}y)\|} & \|\mbox{Re}(e^{i\theta}y)\|
\end{bmatrix}\right\|
\\& \leq \left\|\begin{bmatrix}
\displaystyle{\sup_{\theta \in \mathbb{R}}}\|\mbox{Re}(e^{i\theta}x)\| &
\displaystyle{\sup_{\theta \in \mathbb{R}}}\sqrt{\|\mbox{Re}(e^{i\theta}x)\mbox{Re}(e^{i\theta}y)\|}\\
\displaystyle{\sup_{\theta \in \mathbb{R}}}\sqrt{\|\mbox{Re}(e^{i\theta}x)\mbox{Re}(e^{i\theta}y)\|} &
\displaystyle{\sup_{\theta \in \mathbb{R}}}\|\mbox{Re}(e^{i\theta}y)\|
\end{bmatrix}\right\|
\\&(\mbox{by the norm monotonicity of matrices with nonnegative entries})
\\& = \left\|\begin{bmatrix}
v(x) & \displaystyle{\sup_{\theta \in \mathbb{R}}}\sqrt{\|\mbox{Re}(e^{i\theta}x)\mbox{Re}(e^{i\theta}y)\|}\\
\displaystyle{\sup_{\theta \in \mathbb{R}}}\sqrt{\|\mbox{Re}(e^{i\theta}x)\mbox{Re}(e^{i\theta}y)\|} & v(y)
\end{bmatrix}\right\|
\\& \qquad \qquad \qquad \qquad \qquad \qquad \qquad(\mbox{by Theorem \ref{T.2.2} (i)})
\\ & = \frac{1}{2} \big(v(x) + v(y)\big)
\\& \qquad \qquad  + \frac{1}{2} \sqrt{(v(x) - v(y))^2
+ 4 \displaystyle{\sup_{\theta \in \mathbb{R}}}\|\mbox{Re}(e^{i\theta}x)\mbox{Re}(e^{i\theta}y)\|}
\end{align*}
Therefore, for every $\theta \in \mathbb{R}$ we have
\begin{align*}
\|\mbox{Re}(e^{i\theta}(x + y))\| &\leq \frac{1}{2} \big(v(x) + v(y)\big)
\\& \qquad \qquad  + \frac{1}{2} \sqrt{(v(x) - v(y))^2
+ 4 \displaystyle{\sup_{\theta \in \mathbb{R}}}\|\mbox{Re}(e^{i\theta}x)\mbox{Re}(e^{i\theta}y)\|},
\end{align*}
and hence
\begin{align*}
\displaystyle{\sup_{\theta \in \mathbb{R}}}\|\mbox{Re}(e^{i\theta}(x + y))\| &\leq \frac{1}{2} \big(v(x) + v(y)\big)
\\& \qquad \qquad  + \frac{1}{2} \sqrt{(v(x) - v(y))^2
+ 4 \displaystyle{\sup_{\theta \in \mathbb{R}}}\|\mbox{Re}(e^{i\theta}x)\mbox{Re}(e^{i\theta}y)\|}.
\end{align*}
Now, by Lemma \ref{T.2.2} (i) and the above inequality we get
\begin{align}\label{I.2.11.1}
v(x + y) \leq \frac{1}{2} \Big(v(x) + v(y)
+ \sqrt{(v(x) - v(y))^2
+ 4 \displaystyle{\sup_{\theta \in \mathbb{R}}}\|\mbox{Re}(e^{i\theta}x)\mbox{Re}(e^{i\theta}y)\|}\Big).
\end{align}
Furthermore, by Lemma \ref{T.2.2} (i) we have
\begin{align*}
\displaystyle{\sup_{\theta \in \mathbb{R}}}\|\mbox{Re}(e^{i\theta}x)\mbox{Re}(e^{i\theta}y)\|
\leq \displaystyle{\sup_{\theta \in \mathbb{R}}}\|\mbox{Re}(e^{i\theta}x)\|
\displaystyle{\sup_{\theta \in \mathbb{R}}}\|\mbox{Re}(e^{i\theta}y)\| = v(x)v(y).
\end{align*}
Thus the inequalities (i) follow from (\ref{I.2.11.1}) and the above inequality.

(ii) It is enough to replace $x$ and $y$ in (i) by $ix$ and $iy$, respectively.
\end{proof}
\begin{corollary}\label{C.2.12}
Let $\mathfrak{A}$ be a $C^*$-algebra and let $x, y\in \mathfrak{A}$.
If $x\,{\parallel}_v \,y$ then there exists ${\theta}_0 \in \mathbb{R}$ such that
the following statements hold.
\begin{itemize}
\item[(i)] $\displaystyle{\sup_{\theta \in \mathbb{R}}}\big\|\mbox{Re}(e^{i\theta}x)\mbox{Re}(e^{i(\theta + {\theta}_0)} y)\big\|
= \displaystyle{\sup_{\theta \in \mathbb{R}}}\|\mbox{Re}(e^{i\theta}x)\|\|\mbox{Re}(e^{i\theta}y)\|$.
\item[(ii)] $\displaystyle{\sup_{\theta \in \mathbb{R}}}\big\|\mbox{Im}(e^{i\theta}x)\mbox{Im}(e^{i(\theta + {\theta}_0)} y)\big\|
= \displaystyle{\sup_{\theta \in \mathbb{R}}}\|\mbox{Im}(e^{i\theta}x)\|\|\mbox{Im}(e^{i\theta}y)\|$.
\end{itemize}
\end{corollary}
\begin{proof}
Since $x\,{\parallel}_v \,y$, so there exists ${\theta}_0 \in \mathbb{R}$
such that
\begin{align*}
v(x + e^{i{\theta}_0}y) = v(x) + v(y) = v(x) + v(e^{i{\theta}_0}y).
\end{align*}
Hence by Theorem \ref{T.2.11} it follows that
\begin{align*}
\sup_{\theta \in \mathbb{R}}\big\|\mbox{Re}(e^{i\theta}x)\mbox{Re}(e^{i(\theta + {\theta}_0)} y)\big\| = v(x)v(y)
\end{align*}
and
\begin{align*}
\sup_{\theta \in \mathbb{R}}\big\|\mbox{Im}(e^{i\theta}x)\mbox{Im}(e^{i(\theta + {\theta}_0)} y)\big\| = v(x)v(y).
\end{align*}
These, together with Theorem \ref{T.2.2}, imply that (i) and (ii).
\end{proof}
In the following result we characterize the numerical radius parallelism for elements of a $C^*$-algebra.
\begin{theorem}\label{T.2.13}
Let $\mathfrak{A}$ be a $C^*$-algebra and let $x, y\in \mathfrak{A}$. Then the following statements are equivalent:
\begin{itemize}
\item[(i)] $x\,{\parallel}_v \,y$.
\item[(ii)] There exists a pure state $\varphi$ on $\mathfrak{A}$ such that $|\varphi(x)\varphi(y)| = v(x)v(y)$.
\end{itemize}
\end{theorem}
\begin{proof}
(i)$\Rightarrow$(ii) Let $x\,{\parallel}_v \,y$. Thus $v(x + \lambda y) = v(x) + v(y)$
for some $\lambda\in\mathbb{T}$. Therefore, there exists a pure state $\varphi$
on $\mathfrak{A}$ such that $|\varphi(x + \lambda y)| = v(x + \lambda y)$.
From this it follows that
\begin{align*}
v(x) + v(y) = v(x + \lambda y) &= |\varphi(x + \lambda y)|
\\& \leq |\varphi(x)| + |\varphi(y)| \leq v(x) + |\varphi(y)| \leq v(x) + v(y),
\end{align*}
and hence $|\varphi(x)| = v(x)$ and $|\varphi(y)| = v(y)$. Thus $|\varphi(x)\varphi(y)| = v(x)v(y)$.

(ii)$\Rightarrow$(i) Suppose (ii) holds.
We may assume that $|\varphi(x)\varphi(y)| \neq 0$ otherwise (i) trivially holds.
Put $\lambda = \frac{\varphi(x)\overline{\varphi(y)}}{|\varphi(x)\varphi(y)|}$.
Here, $\overline{\varphi(y)}$ denotes the complex conjugate of $\varphi(y)$.
Since
\begin{align*}
v(x)v(y) = |\varphi(x)\varphi(y)| \leq |\varphi(x)|v(y)\leq v(x)v(y),
\end{align*}
we have $v(x) = |\varphi(x)|$ and so $v(y) = |\varphi(y)|$.
Therefore,
\begin{align*}
v(x) + v(y) &= |\varphi(x)| + |\varphi(y)|
\\& = \left||\varphi(x)| + \frac{\overline{\varphi(y)}}{|\varphi(y)|}\varphi(y)\right|
\\& = \left|\varphi(x) + \frac{\varphi(x)\overline{\varphi(y)}}{|\varphi(x)\varphi(y)|}\varphi(y)\right|
\\& = |\varphi(x + \lambda y)|
\\& \leq v(x + \lambda y) \leq v(x) + v(y).
\end{align*}
This implies that $v(x + \lambda y) = v(x) + v(y)$ and hence $x\,{\parallel}_v \,y$.
\end{proof}
As a consequence of the preceding theorem, we have the following result.
\begin{corollary}\label{C.2.14}
Let $\mathfrak{A}$ be a $C^*$-algebra with identity $e$.
Then for every $x\in \mathfrak{A}$, $x\,{\parallel}_v \,e$.
\end{corollary}
\begin{proof}
Let $x\in \mathfrak{A}$. Thus there exists a pure state $\varphi$ on $\mathfrak{A}$ such that $|\varphi(x)| = v(x)$
and so $|\varphi(x)\varphi(e)| = |\varphi(x)| = v(x) = v(x)v(e)$.
Therefore, Theorem \ref{T.2.13} tells us that $x\,{\parallel}_v \,e$.
\end{proof}
As an immediate consequence of Theorem \ref{T.2.13}, Lemma \ref{L.2.1} and Theorem \ref{T.2.2}, we have the following result.
\begin{corollary}\label{C.2.15}
Let $\mathfrak{A}$ be a $C^*$-algebra and let $x, y\in \mathfrak{A}$.
If $x\,{\parallel}_v \,y$ then the following statements hold.
\begin{itemize}
\item[(i)] There exists a pure state $\varphi$ on $\mathfrak{A}$ such that
\begin{align*}
\displaystyle{\sup_{\theta \in \mathbb{R}}}\big|\mbox{Re}\big(e^{i\theta}\varphi(x)\big)\big|
\big|\mbox{Re}\big(e^{i\theta}\varphi(y)\big)\big|
= \displaystyle{\sup_{\theta \in \mathbb{R}}}\big\|\mbox{Re}(e^{i\theta}x)\big\|\big\|\mbox{Re}(e^{i\theta}y)\big\|.
\end{align*}
\item[(ii)] There exists a pure state $\varphi$ on $\mathfrak{A}$ such that
\begin{align*}
\displaystyle{\sup_{\theta \in \mathbb{R}}}\big|\mbox{Im}\big(e^{i\theta}\varphi(x)\big)\big|
\big|\mbox{Im}\big(e^{i\theta}\varphi(y)\big)\big|
= \displaystyle{\sup_{\theta \in \mathbb{R}}}\big\|\mbox{Im}(e^{i\theta}x)\big\|\big\|\mbox{Im}(e^{i\theta}y)\big\|.
\end{align*}
\end{itemize}
\end{corollary}
We closed this paper with the following equivalence theorem.
In fact, our next result is a characterization of left or right homogenous for
the numerical radius parallelism in unital $C^*$-algebras.
\begin{theorem}
Let $\mathfrak{A}$ be a $C^*$-algebra with identity $e$
and let $c\in \mathcal{Z}(\mathfrak{A})\cap \mathcal{U}(\mathfrak{A})$.
Then for every $x, y \in \mathfrak{A}$ the following statements are equivalent:
\begin{itemize}
\item[(i)] $x\,{\parallel}_v \,y$.
\item[(ii)] $cx\,{\parallel}_v \,cy$.
\item[(iii)] $xc\,{\parallel}_v \,yc$.
\end{itemize}
\end{theorem}
\begin{proof}
Firstly, we show that $v(cz) = v(z) = v(zc)$ for all $z \in \mathfrak{A}$.
Let $\varphi$ be a pure state on $\mathfrak{A}$.
By \cite[Proposition 2.4.4]{Dix} there exist a Hilbert space $\mathscr{H}$,
an irreducible representation $\pi: \mathfrak{A}\rightarrow \mathbb{B}(\mathscr{H})$
and a unit vector $\xi\in \mathscr{H}$ such that for any $z\in\mathfrak{A}$
we have $\varphi(z)=\langle\pi(z)\xi, \xi\rangle.$
Since $c\in \mathcal{Z}(\mathfrak{A})$, by \cite[Proposition II.6.4.13]{Bl},
there exists $\alpha\in\mathbb{C}\setminus\{0\}$ such that $\pi(c) = \alpha e$.
Now from $c\in \mathcal{U}(\mathfrak{A})$ it follows that
\begin{align*}
|\alpha| = \|\pi(c)\| = \sup_{\|\xi\| = 1}\big\|\pi(c)\xi\big\|
= \sup_{\|\xi\| = 1}\sqrt{\langle\pi(c)\xi, \pi(c)\xi\rangle}
= \sup_{\|\xi\| = 1}\sqrt{\langle\pi(c^*c)\xi, \xi\rangle} = 1.
\end{align*}
Therefore for any $z\in\mathfrak{A}$ we obtain
\begin{align*}
|\varphi(cz)| = \big|\langle\pi(cz)\xi, \xi\rangle\big| = \big|\langle\pi(c)\pi(z)\xi, \xi\rangle\big|
= \big|\langle\alpha\pi(z)\xi, \xi\rangle\big| = \big|\langle\pi(z)\xi, \xi\rangle\big| = \big|\varphi(z)\big|.
\end{align*}
From this it follows that
\begin{align*}
v(cz) = \sup_{\varphi \in \mathcal{P}(\mathfrak{A})}\big|\varphi(cz)\big|
= \sup_{\varphi \in \mathcal{P}(\mathfrak{A})}\big|\varphi(z)\big| = v(z),
\end{align*}
and hence $v(cz) = v(z)$.

By using a similar argument we conclude that $v(z) = v(zc)$.

Now, let $x, y \in \mathfrak{A}$. Hence $x\,{\parallel}_v \,y$ if and only if $v(x + \lambda y) = v(x) + v(y)$
for some $\lambda\in\mathbb{T}$,
or equivalently, if and only if
$v\big(c(x + \lambda y)\big) = v(cx) + v(cy)$.
This holds if and only if
$v(cx + \lambda cy) = v(cx) + v(cy)$
for some $\lambda\in\mathbb{T}$,
or equivalently, if and only if
$cx\,{\parallel}_v \,cy$. Therefore, (i)$\Leftrightarrow$(ii).

The proof of the equivalence (i)$\Leftrightarrow$(iii) is similar, so we omit it.
\end{proof}
\bibliographystyle{amsplain}

\end{document}